\documentclass[11pt]{article}
\usepackage{amscd}
\usepackage{amsfonts}
\usepackage{amsmath}
\usepackage{amssymb}
\usepackage{amsthm}
\usepackage{bbm}
\usepackage{CJK}
\usepackage{fancyhdr}
\usepackage{graphicx}
\usepackage{indentfirst}
\usepackage{latexsym,bm}
\usepackage{mathrsfs}
\usepackage{xypic}
\allowdisplaybreaks[4]
\newtheorem{theorem}{Theorem}[section]

\newtheorem{definition}[theorem]{Definition}
\newtheorem{proposition}[theorem]{Proposition}
\newtheorem{example}[theorem]{Example}

\newtheorem{corollary}[theorem]{Corollary}
\newtheorem{remark}[theorem]{Remark}

\usepackage[top=1in,bottom=1in,left=1.25in,right=1.25in]{geometry}
\def\<{\langle}
\def\>{\rangle}

\def\tl{\triangleleft}
\def\tr{\triangleright}

\def\D{\Delta}

\def\l{\ltimes}
\def\o{\otimes}

\def\r{\rho}
\def\na{\natural}

\date{}
\begin{document}
\renewcommand{\baselinestretch}{1.2}
\renewcommand{\arraystretch}{1.0}
\title{\bf Constructions of Rota-Baxter operators by L-R smash products}
 \date{}
\author {{\bf Daowei Lu \footnote {Corresponding author: ludaowei620@126.com.}, Dingguo Wang}\\
{\small $^1$School of Mathematics and Big Data, Jining University}\\
{\small Qufu, Shandong 273155, P. R. China}\\
{\small $^2$School of Mathematical Sciences, Qufu Normal University}\\
{\small Qufu, Shandong 273165, P. R. China}
}
 \maketitle
\begin{center}
\begin{minipage}{12.cm}

\noindent{\bf Abstract.} Let $A$ and $H$ be two cocommutative Hopf algebras such that $A$ is an $H$-bimodule Hopf algebra. Suppose that $R:A\rightarrow A$ is a linear map and $B$ is a  Rota-Baxter operator of $H$. In this paper we will characterize the Rota-Baxter operators on the L-R smash product $A\na H$ and give the necessary and sufficient conditions to make $\overline{B}$ a Rota-Baxter operator of $A\na H$. Then we will consider the dual case, and construct a Rota-Baxter co-operator on the L-R smash coproduct $C\ltimes H$, where $C$ and $H$ are commutative Hopf algebras and $C$ is an $H$-bicomodule Hopf algebra.
\\

\\

 \end{minipage}
 \end{center}
 \normalsize\vskip1cm

\section*{Introduction}
\setcounter{equation} {0} \hskip\parindent

Rota-Baxter operators for associative algebras was firstly introduced by G. Baxter in \cite{Bax} as a tool for studying integral operators in the theory of probability and mathematical statistics. An algebra endowed with a Rota-Baxter operator is called a Rota-Baxter algebra, which has been
related to many topics in mathematical physics. For example, the Rota-Baxter operator on an associative algebra satisfies the famous (operator form of) classical Yang-Baxter equation on the Lie algebra which is the commutator
of the associative algebra\cite{Ag1,Ag2,Ag3}. Rota-Baxter algebras appeared in connection with the work of Connes and Kreimer on renormalization theory in perturbative quantum field theory \cite{CK1,CK2}. It is also related to Loday’s dendriform algebras \cite{Lo, LR}.

The notion of Rota-Baxter operators on Lie algebras was introduced independently by A.A. Belavin, V.G. Drinfeld \cite{BeDr} and M.A. Semenov-Tyan-Shanskii \cite{STS} in their research of the solutions of the classical Yang-Baxter equation. Let $(L,[,])$ be a Lie algebra. A Rota-Baxter operators of weight 1 on $L$ is a linear map $B:L\rightarrow L$ such that for all $x,y\in L,$
$$[B(x),B(y)]=B([x,B(y)]+[B(x),y]+[x,y]).$$

It turns out that on quadratic Lie algebras skew-symmetric solutions of the classical Yang-Baxter equation are in one to one correspondence with skew-symmetric Rota-Baxter operators.

Recently the notion of Rota-Baxter operators on groups was introduced in \cite{GLS}. Let $G$ be a group, a map $B:G\rightarrow G$ is called a Rota-Baxter operator of $G$ if for all $g,h\in G$, 
$$B(g)B(h)=B(gB(g)hB(g)^{-1}).$$

A group $G$ with a Rota-Baxter operator $B$ is called a Rota-Baxter group.  It was pointed out in \cite{GLS} that if $(G, B)$ is a Rota-Baxter Lie group, then the tangent map of $B$ at the identity is a Rota-Baxter operator of weight 1 on the Lie algebra of the Lie group $G$. Later in \cite{Gon} the notion of  Rota-Baxter group was naturally generalized to the notion of Rota-Baxter Hopf algebra, which is a cocommutative Hopf algebra $H$ with a coalgebra map $B:H\rightarrow H$ satisfying for all $g,h\in H$,
$$B(g)B(h)=B(g_1B(g_2)hS(B(g_3))).$$ 

Since any pointed cocommutative Hopf algebra $H$ over an algebraically closed field $F$ is isomorphic to the smash product $U(L)\#F[G]$, where $G$ is a group, and $U(L)$ the universal enveloping algebra of a Lie algebra $L$. Inspired by this result, in \cite{Zhu} the authors presented how to construct Rota-Baxter operators on cocommutative smash product  Hopf algebras. Motivated by this result, since L-R smash product is an important generalization of smash product, in this paper we will construct Rota-Baxter operators through the L-R smash product Hopf algebras. Explicitly let $A$ and $H$ be two cocommutative Hopf algebra such that $(A,\tl,\tr)$ is an $H$-bimodule Hopf algebra. Assume that $R:A\rightarrow A$ is a linear map and $B$ is a Rota-Baxter operator on $H$. Define a linear map $\overline{B}:A\na H\rightarrow A\na H$ by
$$\overline{B}(a\na h)=B(h_1)\tr R(a)\tl B(h_2)\na B(h_3),$$
for all $a\in A,h\in H$. We will give the  necessary and sufficient conditions to make $\overline{B}$ a Rota-Baxter operator of $A\na H$. Finally assume that $C$ and $H$ are commutative Hopf algebras and $C$ is an $H$-bicomodule Hopf algebra. We will examine the dual case and characterize the Rota-Baxter co-operator on  the L-R smash coproduct Hopf algebras $C\ltimes H$.

The paper is organized as follows. In Section 1 we recall some basic and necessary definitions and results on L-R smash product and smash coproduct. In Section 2 we investigate how to construct Rota-Baxter operators on L-R smash product  Hopf algebras and give the necessary and sufficient conditions to describe the Rota-Baxter operators on L-R smash product. In section 3, we will study the dual case and give the  necessary and sufficient conditions to construct a Rota-Baxter co-operator on  L-R smash coproduct  Hopf algebras.

\section{Preliminaries}
\def\theequation{1.\arabic{equation}}
\setcounter{equation} {0} 

Throughout this paper, let $k$ be a fixed field, and all vector spaces and tensor product are over $k$. For a coalgebra $C$, we will use the Heyneman-Sweedler's notation $\Delta(c)=  c_{1}\otimes c_{2},$ for any $c\in C$ (summation omitted).

Let $H$ be a bialgebra and $A$ a bialgebra. Then $A$ is called an $H$-bimodule bialgebra if $A$ is an $H$-bimodule with the left action $\tr:H\o A\rightarrow A$ and the right action $\tl:A\o H\rightarrow A$ such that for all $h\in H,a,b\in A$,
\begin{align*}
\left.
 \begin{aligned}
 &h\tr (ab)=(h_1\tr a)(h_2\tr b),\ h\tr 1_A=\varepsilon(h)1_A, \\
 &(ab)\tl h=( a\tl h_1)(b\tl h_2),\ 1_A\tl h=\varepsilon(h)1_A,
 \end{aligned}
 \right\}&  \text{bimodule algebra}\\
 \left.
 \begin{aligned}
 &(h\tr a)_1\o (h\tr a)_2=h_1\tr a_1\o h_2\tr a_2,\\
&(a\tl h)_1\o (a\tl h)_2= a_1\tl h_1\o  a_2\tl h_2,\\
&\varepsilon(h\tr a)=\varepsilon(a\tl h)=\varepsilon(a)\varepsilon(h).
 \end{aligned}
 \right\}&  \text{bimodule coalgebra}
\end{align*}

Let $(H,S_H)$ be a Hopf algebra and $A$ an $H$-bimodule algebra. Then we have the L-R smash product $A\na H$ (which is equal to $A\o H$ as a vector space) with the following multiplication
$$(a\na h)(b\na g)=(a\tl g_2)(h_1\tr b)\na h_2g_1,$$
and the unit $1_A\na 1_H$ for all $a,b\in A,g,h\in H$. Furthermore if $(A,S_A)$ is a Hopf algebra and $A$ is an $H$-bimodule bialgebra, then $A\na H$ becomes a Hopf algebra with the comultiplication, counit and antipode given by
\begin{align*}
&\D(a\na h)=a_1\na h_1\o a_2\na h_2,\\
& \varepsilon(a\na h)=\varepsilon_A(a)\varepsilon_H(h),\\
&S(a\na h)=S_H(h_3)\tr S_A(a)\tl S_H(h_2)\na S_H(h_1),
\end{align*}
under the following conditions
\begin{align}
h_1\tr a\o h_2=h_2\tr a\o h_1,\label{1a}\\
a\tl h_1\o h_2=a\tl h_2\o h_1.\label{1b}
\end{align}
Note that when $H$ is cocommutative, the identities (\ref{1a}) and (\ref{1b}) naturally hold.

We call $A$ an $H$-bimodule Hopf algebra if $S_A$ is $H$-bilinear.

A bialgebra $C$ is called an $H$-bicomodule bialgebra if $C$ is an $H$-bicomodule with the left coaction $\r^l:C\rightarrow H\o C, \ c\mapsto c_{(-1)}\o c_{(0)}$ and the right coaction $\r^r:C\rightarrow C\o H, \ c\mapsto c_{[0]}\o c_{[1]}$ such that for all $c\in C$,
\begin{align*}
\left.
 \begin{aligned}
&(ab)_{(-1)}\o(ab)_{(0)}=a_{(-1)}b_{(-1)}\o a_{(0)}b_{(0)},\\
&(ab)_{[0]}\o(ab)_{[1]}=a_{[0]}b_{[0]}\o a_{[1]}b_{[1]},\\
&1_{A(-1)}\o 1_{A(0)}=1_H\o1_A,\ 1_{A[0]}\o 1_{A[1]}=1_A\o1_H,
 \end{aligned}
 \right\}&  \text{bicomodule algebra}\\
\left.
 \begin{aligned}
&c_{(-1)}\o c_{(0)1}\o c_{(0)2}=c_{1(-1)}c_{2(-1)}\o c_{1(0)}\o c_{2(0)},\\
&c_{[0]1}\o c_{[0]2}\o c_{[1]}=c_{1[0]}\o c_{2[0]}\o c_{1[1]}c_{2[1]},\\
&\varepsilon(c_{(0)})c_{(-1)}=\varepsilon(c)1_H,\ \varepsilon(c_{[0]})c_{[1]}=\varepsilon(c)1_H.
 \end{aligned}
 \right\}&  \text{bicomodule coalgebra}
\end{align*}

Assume that a Hopf algebra $(C,S_C)$ is an $H$-bicomodule bialgebra and denote $C\o H$ by $C\ltimes H$ and the elements $c\o h$ by $c\ltimes h$. Furthermore assume that 
\begin{align}
&c_{(-1)}h\o c_{(0)}=hc_{(-1)}\o c_{(0)},\label{1c}\\
&c_{[0]}\o c_{[1]}h=c_{[0]}\o hc_{[1]},\label{1d}
\end{align}
for all $c\in C,h\in H$. Then $C\ltimes H$ becomes a Hopf algebra (called L-R smash coproduct) with the following structures:
\begin{align*}
&(c\ltimes g)(d\ltimes h)=cd\ltimes gh,\ 1_{C\ltimes H}=1_C\ltimes 1_H,\\
&\Delta(c\ltimes h)=(c_{1[0]}\ltimes c_{2(-1)}h_1)\o(c_{2(0)}\ltimes h_2c_{1[1]}),\ \varepsilon(c\ltimes h)=\varepsilon(c)\varepsilon(h)\\
&S(c\ltimes h)=S_C(c_{(0)[0]})\ltimes S_H(c_{(-1)}c_{(0)[1]}h).
\end{align*}
Note that when $H$ is commutative, the identities (\ref{1c}) and (\ref{1d}) naturally hold. We call $C$ an $H$-bicomodule Hopf algebra if $S_C$ is $H$-bicolinear.

\section{Rota-Baxter operator on L-R smash product}
\def\theequation{2.\arabic{equation}}
\setcounter{equation} {0} 

In this section, we will construct a Rota-Baxter operator on L-R smash product, and give the necessary and sufficient conditions.

\begin{definition}\cite{Gon} Let $(H,S)$ be a cocommutative Hopf algebra. A coalgebra map $B:H\rightarrow H$ is called a Rota-Baxter operator on $H$ if for all $x,y\in H$,
$$B(x)B(y)=B(x_1B(x_2)yS(B(x_3))).$$
\end{definition}

\begin{theorem}
Let $A$ and $H$ be two cocommutative Hopf algebra such that $(A,\tl,\tr)$ is an $H$-bimodule Hopf algebra. Assume that $R:A\rightarrow A$ is a linear map and $B$ is a Rota-Baxter operator on $H$. Define a linear map $\overline{B}:A\na H\rightarrow A\na H$ by
$$\overline{B}(a\na h)=B(h_1)\tr R(a)\tl B(h_2)\na B(h_3),$$
for all $a\in A,h\in H$. Then $\overline{B}$ is a Rota-Baxter operator on $A\na H$ if and only if $R$ is a Rota-Baxter operator on $A$ and satisfies the following conditions
\begin{align}
&R[(a_1R(a_2)\tl h_1)b(h_2\tr S_A(R(a_3)))]=(S_H(B(h))\tr R(a))R(b),\label{2a}\\
&R(g_{1}B(g_2)\tr b\tl S_H(B(g_3)))=R(b)\tl S_H(B(g)).\label{2b}
\end{align}
\end{theorem}

\begin{proof}
For all $a,b\in A,g,h\in H$, we have 
\begin{align*}
  &\overline{B}(a\na g) \overline{B}(b\na h) \\
  &=[B(g_1)\tr R(a)\tl B(g_2)\na B(g_3)][B(h_1)\tr R(b)\tl B(h_2)\na B(h_3)] \\
  &=(B(g_1)\tr R(a)\tl B(g_2)B(h_1))(B(g_3) B(h_2)\tr R(b)\tl B(h_3))\na B(g_4)B(h_4),
\end{align*}
and 
\begin{align*}
&\overline{B}[(a_1\na g_1)\overline{B}(a_2\na g_2)(b\na h)S(\overline{B}(a_3\na g_3))]\\
&=\overline{B}[(a_1\na g_1)(B(g_2)\tr R(a_2)\tl B(g_3)\na B(g_4))(b\na h)S(B(g_5)\tr R(a_3)\tl B(g_6)\na B(g_7))]\\
&=\overline{B}[(a_1\na g_1)(B(g_2)\tr R(a_2)\tl B(g_3)\na B(g_4))(b\na h)\\
&\quad(S_H(B(g_7))B(g_5)\tr S_A(R(a_3))\tl B(g_6)S_H(B(g_8)))\na S_H(B(g_9))]\\
&=\overline{B}[(a_1\na g_1)(B(g_2)\tr R(a_2)\tl B(g_3)\na B(g_4))(b\na h)(S_A(R(a_3))\na S_H(B(g_5)))]\\
&=\overline{B}[((a_1\tl B(g_6))(g_{1}B(g_2)\tr R(a_2)\tl B(g_3))\na g_{4}B(g_5)))(b\na h)(S_A(R(a_3))\na S_H(B(g_7)))]\\
&=\overline{B}[((a_1\tl B(g_8)h_1)(g_{1}B(g_2)\tr R(a_2)\tl B(g_3)h_2)(g_{4}B(g_5)\tr b)\na  g_{6}B(g_7)h_3)\\
&\quad(S_A(R(a_3))\na S_H(B(g_9)))]\\
&=\overline{B}[((a_1\tl B(g_{10})h_1S_H(B(g_{11})))(g_{1}B(g_2)\tr R(a_2)\tl B(g_3)h_2S_H(B(g_{12})))\\
&\quad(g_{4}B(g_5)\tr b
\tl S_H(B(g_{13})))(g_{6}B(g_7)h_3\tr S_A(R(a_3)))\na  g_{8}B(g_9)h_4S_H(B(g_{14}))]\\
&=B(g_{8}B(g_9)h_4S_H(B(g_{18})))\tr R[(a_1\tl B(g_{14})h_1S_H(B(g_{15})))(g_{1}B(g_2)\tr R(a_2)\tl B(g_3)h_2S_H(B(g_{16})))\\
&\quad(g_{4}B(g_5)\tr b
\tl S_H(B(g_{17})))(g_{6}B(g_7)h_3\tr S_A(R(a_3)))]\tl B(g_{10}B(g_{11})h_5S_H(B(g_{19})))\\
&\quad\na  B(g_{12}B(g_{13})h_6S_H(B(g_{20})))\\
&=B(g_{8}B(g_9)h_4S_H(B(g_{10})))\tr R[(a_1\tl B(g_{17})h_1S_H(B(g_{18})))(g_{1}B(g_2)\tr R(a_2)\tl B(g_3)h_2S_H(B(g_{19})))\\
&\quad(g_{4}B(g_5)\tr b
\tl S_H(B(g_{20})))(g_{6}B(g_7)h_3\tr S_A(R(a_3)))]\tl B(g_{11}B(g_{12})h_5S_H(B(g_{13})))\\
&\quad\na  B(g_{14}B(g_{15})h_6S_H(B(g_{16})))\\
&=B(g_{8})B(h_4)\tr R[(a_1\tl B(g_{11})h_1S_H(B(g_{12})))(g_{1}B(g_2)\tr R(a_2)\tl B(g_3)h_2S_H(B(g_{13})))\\
&\quad(g_{4}B(g_5)\tr b
\tl S_H(B(g_{14})))(g_{6}B(g_7)h_3\tr S_A(R(a_3)))]\tl B(g_{9})B(h_5))\na  B(g_{10})B(h_6)\\
&=B(g_{1})B(h_1)\tr R[(a_1\tl B(g_{2})h_2S_H(B(g_{3})))(g_{4}B(g_5)\tr R(a_2)\tl B(g_6)h_3S_H(B(g_{7})))\\
&\quad(g_{8}B(g_9)\tr b
\tl S_H(B(g_{10})))(g_{11}B(g_{12})h_4\tr S_A(R(a_3)))]\tl B(g_{13})B(h_5))\na  B(g_{14})B(h_6).
\end{align*}
$(\Rightarrow):$ Assume that $\overline{B}$ is a Rote-Baxter operator on $A\na H$. Since $\overline{B}$ is a coalgebra map, we have that 
$$\D(\overline{B}(a\na1))=\overline{B}(a_1\na1)\overline{B}(a_2\na1).$$
Hence $\D(R(a))=R(a_1)\o R(a_2)$, which means that $R$ is a coalgebra map. And
\begin{align*}
&B(g_{1})B(h_1)\tr R[(a_1\tl B(g_{2})h_2S_H(B(g_{3})))(g_{4}B(g_5)\tr R(a_2)\tl B(g_6)h_3S_H(B(g_{7})))\\
&(g_{8}B(g_9)\tr b
\tl S_H(B(g_{10})))(g_{11}B(g_{12})h_4\tr S_A(R(a_3)))]\tl B(g_{13})B(h_5))\na  B(g_{14})B(h_6)\\
&=(B(g_1)\tr R(a)\tl B(g_2)B(h_1))(B(g_3) B(h_2)\tr R(b)\tl B(h_3))\na B(g_4)B(h_4),
\end{align*}
Applying $id\o \varepsilon$ to the above equation we have
\begin{align*}
&B(g_{1})B(h_1)\tr R[(a_1\tl B(g_{2})h_2S_H(B(g_{3})))(g_{4}B(g_5)\tr R(a_2)\tl B(g_6)h_3S_H(B(g_{7})))\\
&(g_{8}B(g_9)\tr b
\tl S_H(B(g_{10})))(g_{11}B(g_{12})h_4\tr S_A(R(a_3)))]\tl B(g_{13})B(h_5)\\
&=(B(g_1)\tr R(a)\tl B(g_2)B(h_1))(B(g_3) B(h_2)\tr R(b)\tl B(h_3)),
\end{align*}
then 
\begin{align*}
&B(h_1)\tr R[(a_1\tl B(g_{1})h_2S_H(B(g_{2})))(g_{3}B(g_4)\tr R(a_2)\tl B(g_5)h_3S_H(B(g_{6})))\\
&(g_{7}B(g_8)\tr b
\tl S_H(B(g_{9})))(g_{10}B(g_{11})h_4\tr S_A(R(a_3)))]\tl B(g_{12})\\
&=(R(a)\tl B(g))(B(h)\tr R(b)).
\end{align*}
This is equivalent to the following equation
\begin{align}\label{2c}
& R[(a_1\tl B(g_{1})h_1S_H(B(g_{2})))\nonumber\\
&(g_{3}B(g_4)\tr (((R(a_2)\tl B(g_5)h_2S_H(B(g_{6})))(b\tl S_H(B(g_{7})))(h_3\tr S_A(R(a_3)))))]\\
&=(S_H(B(h))\tr R(a))(R(b)\tl S_H(B(g))).\nonumber
\end{align}
Now take $g=1_H$ we obtain 
$$
R[(a_1R(a_2)\tl h_1)b(h_2\tr S_A(R(a_3)))]=(S_H(B(h))\tr R(a))R(b),
$$
Furthermore set $h=1_H$, we have
$$
R(a_1R(a_2)S_A(R(a_3)))=R(a)R(b),
$$
Again in the identity (\ref{2c}), let $a=1_A,h=1_H$ we get
$$R(g_{1}B(g_2)\tr b\tl S_H(B(g_3)))=R(b)\tl S_H(B(g)).$$

$(\Leftarrow):$ First of all, for $g\in H, a\in A$, we have
\begin{align*}
&R(g_1B(g_2)\tr a)\tl B(g_3)\\
&=R(g_1B(g_2)\tr (a\tl B(g_3))\tl S_H(B(g_4)))\tl B(g_5)\\
&\stackrel{(\ref{2b})}{=}R(a\tl B(g_1))\tl S_H(B(g_4))B(g_5)\\
&=R(a\tl B(g)),
\end{align*}
and hence
\begin{align*}
R(a)\tl B(g)&=R(g_1B(g_2)\tr (S_H(g_3B(g_4))\tr a))\tl B(g_5)\\
&=R(g_1B(g_2)\tr (S_H(g_4B(g_5))\tr a))\tl B(g_3)\\
&=R(S_H(g_1B(g_2))\tr a\tl B(g_3)),
\end{align*}
thus
\begin{equation}
R(a)\tl B(g)=R(S_H(g_1B(g_2))\tr a\tl B(g_3)).\label{2d}
\end{equation}

For all $a,b\in A,g,h\in H$, we only need to prove the identity (\ref{2c}).
\begin{align*}
& R[(a_1\tl B(g_{1})h_1S_H(B(g_{2})))\\
&(g_{3}B(g_4)\tr (((R(a_2)\tl B(g_5)h_2S_H(B(g_{6})))(b\tl S_H(B(g_{7})))(h_3\tr S_A(R(a_3))))))]\\
&=R[g_{1}B(g_2)\tr ((S_H(g_{3}B(g_4))\tr a_1\tl B(g_{5})h_1S_H(B(g_{6})))\\
&(R(a_2)\tl B(g_7)h_2S_H(B(g_{8})))(b\tl S_H(B(g_{9})))(h_3\tr S_A(R(a_3))))]\\
&=R[g_{1}B(g_2)\tr ((S_H(g_{4}B(g_5))\tr a_1\tl B(g_{6})h_1)\\
&(R(a_2)\tl B(g_7)h_2)b(h_3\tr S_A(R(a_3))\tl B(g_{8})))\tl S_H(B(g_{3}))]\\
&\stackrel{(\ref{2b})}{=}R[(S_H(g_{1}B(g_2))\tr a_1\tl B(g_{3})h_1)(\underline{R(a_2)\tl B(g_4)}h_2)b(h_3\tr S_A(R(a_3))\tl B(g_{5})))]\tl S_H(B(g_{6}))\\
&\stackrel{(\ref{2d})}{=}R[((S_H(g_{1}B(g_2))\tr a_1\tl B(g_{3}))\tl h_1)(R(S_H(g_{4}B(g_5))\tr a_2\tl B(g_6))\tl h_2)b\\
&(h_3\tr S_A(\underline{R(a_3)\tl B(g_{7})}))]\tl S_H(B(g_{8}))\\
&\stackrel{(\ref{2d})}{=}R[((S_H(g_{1}B(g_2))\tr a_1\tl B(g_{3}))R(S_H(g_{4}B(g_5))\tr a_2\tl B(g_6)))\tl h_1)b\\
&(h_2\tr S_A(R(S_H(g_{7}B(g_8))\tr a_3\tl B(g_{9}))))]\tl S_H(B(g_{10}))\\
&\stackrel{(\ref{2a})}{=}[(S_H(B(h))\tr R(S_H(g_{1}B(g_2))\tr a)\tl B(g_{3}))R(b)]\tl S_H(B(g_{4}))\\
&\stackrel{(\ref{2d})}{=}[(S_H(B(h))\tr R(a)\tl B(g_1))R(b)]\tl S_H(B(g_{2}))\\
&=(S_H(B(h))\tr R(a)) (R(b)\tl S_H(B(g_3)).
\end{align*}
The proof is completed.
\end{proof}

\begin{remark}
In the above theorem, when the right action is trivial, that is, $a\tl h=\varepsilon(h)a$ for all $a\in A,h\in H$, we could recover the main result Theorem 3.3 of \cite{Zhu}.
\end{remark}

\begin{corollary}
Let $A$ and $H$ be two cocommutative Hopf algebra such that $(A,\tl,\tr)$ is an $H$-bimodule Hopf algebra. Assume that $R:A\rightarrow A$ is a linear map. Define a linear map $\overline{B}:A\na H\rightarrow A\na H$ by
$$\overline{B}(a\na h)=S_H(h_1)\tr R(a)\tl S_H(h_2)\na S_H(h_3),$$
for all $a\in A,h\in H$. Then $\overline{B}$ is a Rota-Baxter operator on $A\na H$ if for all $a,b\in A,h\in H$, the  following equation holds
\begin{align}
&R[(a_1R(a_2)\tl h_1)b(h_2\tr S_A(R(a_3)))]=(h\tr R(a))R(b),\label{2e}\\
&R(a\tl h)=R(a)\tl h.\label{2f}
\end{align}

\end{corollary}

\begin{proof}
Since for cocommutative Hopf algebra the square of the antipode coincides with the identity, the result is straightforward.
\end{proof}

\begin{corollary}
Let $A$ and $H$ be two cocommutative Hopf algebra such that $(A,\tl,\tr)$ is an $H$-bimodule Hopf algebra. Assume that $B$ is a Rota-Baxter operator on $H$. Define a linear map $\overline{B}:A\na H\rightarrow A\na H$ by
$$\overline{B}(a\na h)=B(h_1)\tr S_A(a)\tl B(h_2)\na B(h_3),$$
for all $a\in A,h\in H$. Then $\overline{B}$ is a Rota-Baxter operator on $A\na H$ if and only if  for all $a\in A,h\in H$, the following equations hold
\begin{align}
&h\tr S_A(a)=S_H(B(h))\tr S_A(a),\label{2e}\\
&h_{1}B(h_2)\tr a=\varepsilon(h)a.\label{2f}
\end{align}

\end{corollary}

\begin{proof}
The proof is straightforward and left to the reader.
\end{proof}

\begin{example}\label{2-6}
Let $C_2=\{1,g|g^2=1\}$ and $C_3=\{1,h|h^3=1\}$ be the cyclic group of order 2 and 3, respectively. Define the left and right actions of $k[C_2]$ on $k[C_3]$ as follows:
\begin{align*}
&g\tr h=h,\quad g\tr h^2=h^2,\\
&h\tl g=h^2,\quad h^2\tl g=h.
\end{align*}
It is straightforward to verify that $k[C_3]$ is a $k[C_2]$-bimodule Hopf algebra; hence we have the L-R smash product $k[C_3]\na k[C_2]$. And by a routine verification one could find that $k[C_3]\na k[C_2]\cong k[S_3]$ as Hopf algebras, where $S_3$ be the symmetric group on 3 symbols.
\begin{itemize}
  \item [(1)] Let $R:k[C_3]\rightarrow k[C_3]$ be the trivial Rota-Baxter operator, i.e., $R(h)=1$, then the equation (\ref{2d}) is satisfied, and 
  $$\overline{B}(h^i\na g)=S_H(g)\tr R(h^i)\tl S_H(g)\na S_H(g)=1\na g,\  i=1,2.$$
  \item [(2)] Define $R:k[C_3]\rightarrow k[C_3]$ by 
  $$R(h)=h^2,\ R(h^2)=h,$$
  then the equation (\ref{2d}) is also satisfied, and 
 \begin{align*}
&\overline{B}(h\na g)=g\tr R(h)\tl g\na g=g\tr h^2\tl g\na g=h\na g,\\
&\overline{B}(h^2\na g)=g\tr R(h)\tl g\na g=g\tr h\tl g\na g=h^2\na g,
 \end{align*}
which implies that $\overline{B}=id_{k[C_3]\na k[C_2]}$.
\end{itemize}
\end{example}

%
%
%

\section{Rota-Baxter co-operator on L-R smash coproduct}
\def\theequation{3.\arabic{equation}}
\setcounter{equation} {0} 

In this section we will consider the dual case and characterize a class of Rota-Baxter co-operator on L-R smash coproduct.

\begin{definition}\cite{Zheng}
Let $(H,S)$ be a commutative Hopf algebra. An algebra map $B : H\rightarrow H$ is called a Rota-Baxter co-operator on $H$ if for all $x\in H$,
\begin{equation}
B(x_1)\o B(x_2)=B(x)_1B(B(x)_2S(B(x)_4))\o B(x)_3.\label{3a}
\end{equation}
\end{definition}

\begin{proposition}
Let $H$ be a commutative Hopf algebra with a Rota-Baxter co-operator $B$, then $\varepsilon\circ B=\varepsilon$.
\end{proposition}

\begin{proof}
For all $x\in H$, applying $\varepsilon\o\varepsilon$ to both sides of (\ref{3a}), we have
$$\varepsilon(B(x_1))\varepsilon(B(x_2))=\varepsilon(B(x)_2S(B(x)_2))=\varepsilon(B(x)),$$
that is $(\varepsilon\circ B)\ast (\varepsilon\circ B)=\varepsilon\circ B$. Since $B(1)=1$, we get
$$\varepsilon(B(x_1))\varepsilon(B(S(x_2)))=\varepsilon(B(x_1S(x_2))=\varepsilon(x),$$
that is $(\varepsilon\circ B)\ast (\varepsilon\circ B\circ S)=\varepsilon$. Therefore  $\varepsilon\circ B=\varepsilon$.
\end{proof}

\begin{theorem}\label{3-1}
Let $C$ and $H$ be two commutative Hopf algebra such that $(C,\r^l,\r^r)$ is an $H$-bicomodule Hopf algebra. Assume that $R:C\rightarrow C$ is a linear map and $B$ is a Rota-Baxter co-operator on $H$. Define a linear map $\widetilde{B}:C\ltimes H\rightarrow C\ltimes H$ by
$$\widetilde{B}(c\ltimes h)=R(c_{(0)[0]})\ltimes B(c_{(-1)}c_{(0)[1]}h),$$
for all $c\in C,h\in H$. Then $\widetilde{B}$ is a Rota-Baxter co-operator on $C\ltimes H$ if and only if $R$ is a Rota-Baxter co-operator on $C$ and satisfies the following conditions
\begin{align}
&R(c)_{1[0]1}R(R(c)_{1[0]2}S_C(R(c)_{3(0)}))\o R(c)_2\o R(c)_{1[1]}R(c)_{3(-1)}\nonumber\\
&\ =R(c_{1(0)})\o R(c_2)\o B(S_H(c_{1(-1)})),\label{3c}\\
&R(c)_{[0](-1)1}B(R(c)_{[0](-1)2}S_H(R(c)_{[1]}))\o R(c)_{[0](0)}=B(S_H(c_{[1]}))\o R(c_{[0]}).\label{3d}
\end{align}
\end{theorem}

\begin{proof}
For all $c\in C,h\in H$, we firstly compute
\begin{align*}
&(c\ltimes h)_1\o(c\ltimes h)_2\o(c\ltimes h)_3\o(c\ltimes h)_4\\
&=(c_{1[0]}\ltimes c_{2(-1)}h_1)_1\o(c_{1[0]}\ltimes c_{2(-1)}h_1)_2\o(c_{2(0)}\ltimes h_2c_{1[1]})_1\o(c_{2(0)}\ltimes h_2c_{1[1]})_2\\
&=\underline{c_{1[0]1[0]}}\ltimes \underline{c_{1[0]2(-1)}}c_{2(-1)1}h_1\o\underline{ c_{1[0]2(0)}}\ltimes c_{2(-1)2}h_2\underline{c_{1[0]1[1]}}\\
&\quad\o c_{2(0)1[0]}\ltimes c_{2(0)2(-1)}h_3\underline{c_{1[1]1}}\o c_{2(0)2(0)}\ltimes h_4\underline{c_{1[1]2}} c_{2(0)1[1]}\\
&=c_{1[0][0]}\ltimes c_{2[0](-1)}c_{3(-1)1}h_1\o c_{2[0](0)}\ltimes c_{3(-1)2}h_2c_{1[0][1]}\\
&\quad\o c_{3(0)1[0]}\ltimes c_{3(0)2(-1)}h_3c_{1[1]1}c_{2[1]1}\o c_{3(0)2(0)}\ltimes h_4c_{1[1]2}c_{2[1]2} c_{3(0)1[1]}\\
&=c_{1[0]}\ltimes c_{2[0](-1)}c_{3(-1)1}c_{4(-1)1}h_1\o c_{2[0](0)}\ltimes c_{3(-1)2}c_{4(-1)2}h_2c_{1[1]1}\\
&\quad\o c_{3(0)[0]}\ltimes c_{4(0)(-1)}h_3c_{1[1]2}c_{2[1]1}\o c_{4(0)(0)}\ltimes h_4c_{1[1]3}c_{2[1]2} c_{3(0)[1]}\\
&=c_{1[0]}\ltimes c_{2[0](-1)}c_{3(-1)1}c_{4(-1)1}h_1\o c_{2[0](0)}\ltimes c_{3(-1)2}c_{4(-1)2}h_2c_{1[1]1}\\
&\quad\o c_{3(0)[0]}\ltimes c_{4(-1)3}h_3c_{1[1]2}c_{2[1]1}\o c_{4(0)}\ltimes h_4c_{1[1]3}c_{2[1]2} c_{3(0)[1]}.
\end{align*}
Then
\begin{align*}
&(c\ltimes h)_1\widetilde{B}((c\ltimes h)_2S((c\ltimes h)_4))\o(c\ltimes h)_3\\
&=(c_{1[0]}\ltimes c_{2[0](-1)}c_{3(-1)1}c_{4(-1)1}h_1)\widetilde{B}((c_{2[0](0)}\ltimes c_{3(-1)2}c_{4(-1)2}h_2c_{1[1]1})S( c_{4(0)}\ltimes h_4c_{1[1]3}c_{2[1]2} c_{3(0)[1]}))\\
&\quad\o  c_{3(0)[0]}\ltimes c_{4(-1)3}h_3c_{1[1]2}c_{2[1]1}\\
&=(c_{1[0]}\ltimes c_{2[0](-1)}c_{3(-1)1}c_{4(-1)1}h_1)\\
&\widetilde{B}( c_{2[0](0)}S_C(c_{4(0)[0]})\ltimes c_{3(-1)2}c_{4(-1)2}h_2c_{1[1]1}S_H(c_{4(-1)4}c_{4(0)[1]}h_4c_{1[1]3}c_{2[1]2} c_{3(0)[1]}))\\
&\quad\o  c_{3(0)[0]}\ltimes c_{4(-1)3}h_3c_{1[1]2}c_{2[1]1}\\
&=(c_{1[0]}\ltimes c_{2[0](-1)1}c_{3(-1)1}c_{4(-1)1}h_1)(R(c_{2[0](0)}S_C(c_{4(0)[0](0)}))\\
&\ltimes B(c_{2[0](-1)2}c_{4(0)[0](-1)}c_{2[1]1}c_{4(0)[1]1} c_{3(-1)2}c_{4(-1)2}h_2c_{1[1]1}S_H(c_{4(-1)4}c_{4(0)[1]2}h_4c_{1[1]3}c_{2[1]3} c_{3(0)[1]})))\\
&\quad\o  c_{3(0)[0]}\ltimes c_{4(-1)3}h_3c_{1[1]2}c_{2[1]2}\\
&=c_{1[0]}R(c_{2[0](0)}S_C(c_{4(0)}))\\
&\ltimes c_{2[0](-1)1}c_{3(-1)1}c_{4(-1)1}h_1 B(c_{2[0](-1)2}c_{3(-1)2}c_{4(-1)2}h_2c_{1[1]1}c_{2[1]1}S_H(h_4c_{1[1]3}c_{2[1]3} c_{3(0)[1]}))\\
&\quad\o  c_{3(0)[0]}\ltimes c_{4(-1)3}h_3c_{1[1]2}c_{2[1]2}\\
&=c_{1[0]1}R(c_{1[0]2(0)}S_C(c_{3(0)}))\\
&\quad\ltimes c_{1[0]2(-1)1}c_{2(-1)1}c_{3(-1)1}h_1 B(c_{1[0]2(-1)2}c_{2(-1)2}c_{3(-1)2}h_2c_{1[1]1}S_H(h_4c_{1[1]3} c_{2(0)[1]}))\\
&\quad\o  c_{2(0)1[0]}\ltimes c_{3(-1)3}h_3c_{1[1]2}\\
&=c_{1[0]1}R(c_{1[0]2(0)}S_C(c_{2(0)2(0)}))
\ltimes c_{1[0]2(-1)1}c_{2(-1)1}h_1 B(c_{1[0]2(-1)2}c_{2(-1)2}h_2c_{1[1]1}S_H(h_4c_{1[1]3} c_{2(0)1[1]}))\\
&\quad\o  c_{2(0)1[0]}\ltimes c_{2(0)2(-1)}h_3c_{1[1]2}\\
&\stackrel{h=B(x)}{=}c_{1[0]1}R(c_{1[0]2(0)}S_C(c_{2(0)2(0)}))
\ltimes c_{1[0]2(-1)1}c_{2(-1)1}B(x_1) B(c_{1[0]2(-1)2}c_{2(-1)2}c_{1[1]1}S_H(c_{1[1]3} c_{2(0)1[1]}))\\
&\quad\o  c_{2(0)1[0]}\ltimes c_{2(0)2(-1)}B(x_2)c_{1[1]2}.
\end{align*}
So on one hand
\begin{align*}
&\widetilde{B}(c\ltimes h)_1\widetilde{B}(\widetilde{B}(c\ltimes h)_2S(\widetilde{B}(c\ltimes h)_4))\o \widetilde{B}(c\ltimes h)_3\\
&=(R(c_{(0)[0]})\ltimes B(c_{(-1)}c_{(0)[1]}h))_1\widetilde{B}((R(c_{(0)[0]})\ltimes B(c_{(-1)}c_{(0)[1]}h))_2S((R(c_{(0)[0]})\ltimes B(c_{(-1)}c_{(0)[1]}h))_4))\\
&\quad\o (R(c_{(0)[0]})\ltimes B(c_{(-1)}c_{(0)[1]}h))_3\\
&=R(c_{(0)[0]})_{1[0]1}R(R(c_{(0)[0]})_{1[0]2(0)}S_C(R(c_{(0)[0]})_{2(0)2(0)}))\ltimes R(c_{(0)[0]})_{1[0]2(-1)1}R(c_{(0)[0]})_{2(-1)1}\\
&\quad\cdot B(c_{(-1)1}c_{(0)[1]1}h_1) B(R(c_{(0)[0]})_{1[0]2(-1)2}R(c_{(0)[0]})_{2(-1)2}R(c_{(0)[0]})_{1[1]1}S_H(R(c_{(0)[0]})_{1[1]3} R(c_{(0)[0]})_{2(0)1[1]}))\\
&\quad\o  R(c_{(0)[0]})_{2(0)1[0]}\ltimes R(c_{(0)[0]})_{2(0)2(-1)}B(c_{(-1)2}c_{(0)[1]2}h_2)R(c_{(0)[0]})_{1[1]2}.
\end{align*}

On the other hand,
\begin{align*}
&\widetilde{B}((c\ltimes h)_1)\o\widetilde{B}((c\ltimes h)_2)\\
&=\widetilde{B}(c_{1[0]}\ltimes c_{2(-1)}h_1)\o \widetilde{B}(c_{2(0)}\ltimes h_2c_{1[1]})\\
&=R(c_{1[0](0)[0]})\ltimes B(c_{1[0](-1)}c_{1[0](0)[1]}c_{2(-1)}h_1)\o R(c_{2(0)(0)[0]})\ltimes B(c_{2(0)(-1)}c_{2(0)(0)[1]}h_2c_{1[1]})\\
&=R(c_{1[0][0](0)})\ltimes B(c_{1[0][0](-1)}c_{1[0][1]}c_{2(-1)1}h_1)\o R(c_{2(0)[0]})\ltimes B(c_{2(-1)2}c_{2(0)[1]}h_2c_{1[1]})\\
&=R(c_{1[0](0)})\ltimes B(c_{1[0](-1)}c_{1[1]1}c_{2(-1)1}h_1)\o R(c_{2(0)[0]})\ltimes B(c_{2(-1)2}c_{2(0)[1]}h_2c_{1[1]2}).
\end{align*}
Therefore the identity $\widetilde{B}(c\ltimes h)_1\widetilde{B}(\widetilde{B}(c\ltimes h)_2S(\widetilde{B}(c\ltimes h)_4))\o \widetilde{B}(c\ltimes h)_3=\widetilde{B}((c\ltimes h)_1)\o\widetilde{B}((c\ltimes h)_2)$ implies
\begin{align*}
&R(c_{(0)[0]})_{1[0]1}R(R(c_{(0)[0]})_{1[0]2(0)}S_C(R(c_{(0)[0]})_{2(0)2(0)}))\ltimes R(c_{(0)[0]})_{1[0]2(-1)1}R(c_{(0)[0]})_{2(-1)1}\\
&\cdot B(c_{(-1)1}c_{(0)[1]1}h_1) B(R(c_{(0)[0]})_{1[0]2(-1)2}R(c_{(0)[0]})_{2(-1)2}R(c_{(0)[0]})_{1[1]1}S_H(R(c_{(0)[0]})_{1[1]3} R(c_{(0)[0]})_{2(0)1[1]}))\\
&\quad\o  R(c_{(0)[0]})_{2(0)1[0]}\ltimes R(c_{(0)[0]})_{2(0)2(-1)}B(c_{(-1)2}c_{(0)[1]2}h_2)R(c_{(0)[0]})_{1[1]2}\\
&=R(c_{1[0](0)})\ltimes B(c_{1[0](-1)}c_{1[1]1}c_{2(-1)1}h_1)\o R(c_{2(0)[0]})\ltimes B(c_{2(-1)2}c_{2(0)[1]}h_2c_{1[1]2}).
\end{align*}
By a long and tedious computation, the above identity is equivalent to
\begin{align}
&R(c)_{1[0]}R(R(c)_{2(0)1[0]})R(S_C(R(c)_{2(0)3(0)}))\nonumber\\
&\o B(R(c)_{1[1]1}) B(S_H(R(c)_{1[1]3})) R(c)_{2(-1)1} B(R(c)_{2(-1)2}) B(R(c)_{2(0)1[1]1})\nonumber\\
&B(S_H(R(c)_{2(0)1[1]3}))B(S_H(R(c)_{2(0)2[1]}))\nonumber\\
&\o R(c)_{2(0)2[0]}\o R(c)_{1[1]2} R(c)_{2(0)1[1]2} R(c)_{2(0)3(-1)}\label{3b}\\
&=R(c_{1(0)})\o B(S_H(c_{2[1]}))\o R(c_{2[0]})\o B(S_H(c_{1(-1)})).\nonumber
\end{align}

$(\Rightarrow):$ Assume that $\widetilde{B}$ is a Rota-Baxter co-operator on $C\ltimes H$. Then for all $c,d\in C$, we have 
\begin{align*}
\widetilde{B}((c\l 1)(d\l 1))&=R((cd)_{(0)[0]})\l B((cd)_{(-1)}(cd)_{(0)[1]})\\
&=R(c_{(0)[0]}d_{(0)[0]})\l B(c_{(-1)}d_{(-1)}c_{(0)[1]}d_{(0)[1]}),
\end{align*}
and 
$$\widetilde{B}(c\l 1)\widetilde{B}(d\l 1)=R(c_{(0)[0]})R(d_{(0)[0]})\l B(c_{(-1)}c_{(0)[1]})B(d_{(-1)}d_{(0)[1]}).$$
Since $\widetilde{B}$ is an algebra map, 
$$R(c_{(0)[0]}d_{(0)[0]})\l B(c_{(-1)}d_{(-1)}c_{(0)[1]}d_{(0)[1]})=R(c_{(0)[0]})R(d_{(0)[0]})\l B(c_{(-1)}c_{(0)[1]})B(d_{(-1)}d_{(0)[1]}).$$
Applying $id\o\varepsilon$ to the above identity we obtain $R(cd)=R(c)R(d)$, that is, $R$ is a morphism of algebra. 

Then apply $id\o\varepsilon\o id\o\varepsilon$ to (\ref{3b}) and let $h=1$, we have
$$
R(c)_{1}R(R(c)_{2}S_C(R(c)_{4}))\o  R(c)_{3}=R(c_{1})\o R(c_{2}).
$$
Hence $R$ is a Rota-Baxter co-operator on $C$.

Applying $id\o \varepsilon\o id\o id$ to (\ref{3b}) and setting $h=1$, we obtain
\begin{align*}
&R(c)_{1[0]} R(R(c)_{2[0]}) R(S_C(R(c)_{4(0)}))\o R(c)_{3}\o R(c)_{1[1]} R(c)_{2[1]} R(c)_{4(-1)}\\
&=R(c_{1(0)})\o R(c_{2})\o B(S_H(R(c_{1(-1)})))
\end{align*}
which is equivalent to
\begin{align*}
&R(c)_{1[0]1}R(R(c)_{1[0]2}S_C(R(c)_{3(0)}))\o R(c)_2\o R(c)_{1[1]}R(c)_{3(-1)}\\
&=R(c_{1(0)})\o R(c_{2})\o B(S_H(c_{1(-1)})).
\end{align*}
Again applying $\varepsilon\o id\o id\o\varepsilon$ to (\ref{3b}) and letting $h=1$, we have
\begin{align*}
&R(c)_{(-1)1}B(R(c)_{(-1)2})B(S_H(R(c)_{(0)[1]}))\o R(c)_{(0)[0]}\\
&=B(S_H(R(c_{[1]})))\o R(c_{[0]}),
\end{align*}
which is equivalent to the identity $(\ref{3d}).$

$(\Leftarrow):$ Assume that $R$ is a Rota-Baxter co-operator on $C$ and the identities (\ref{3c}) and (\ref{3d}) are satisfied. We only need to verify the equation (\ref{3b}). For all $c\in C,h\in H$, firstly we have
\begin{align*}
&R(c_{[0]})_{(-1)1}B(R(c_{[0]})_{(-1)2})B(c_{[1]})\o R(c_{[0]})_{(0)}\\
&=R(c_{[0]})_{[0][0](-1)1}B(R(c_{[0]})_{[0][1]})B(S_H(R(c_{[0]})_{[1]}))B(R(c_{[0]})_{[0][0](-1)2})B(c_{[1]})\o R(c_{[0]})_{[0][0](0)}\\
&=R(c_{[0]})_{[0](-1)1}B(R(c_{[0]})_{[0](0)[1]})B(S_H(R(c_{[0]})_{[1]}))B(R(c_{[0]})_{[0](-1)2})B(c_{[1]})\o R(c_{[0]})_{[0](0)[0]}\\
&\stackrel{(\ref{3d})}{=}B(S_H(c_{[0][1]}))B(R(c_{[0][0]})_{[1]})B(c_{[1]})\o R(c_{[0][0]})_{[0]}\\
&=B(R(c)_{[1]})\o R(c)_{[0]}.
\end{align*}
and 
\begin{align}
&R(c)_{(0)[0]}\o B(S_H(R(c)_{(-1)1})) S_H(R(c)_{(-1)2})B(R(c)_{(0)[1]})\nonumber\\
&=R(c)_{[0](0)}\o B(S_H(R(c)_{[0](-1)1})) S_H(R(c)_{[0](-1)2})B(R(c)_{[1]})\nonumber\\
&=R(c_{[0]})_{(0)}\o B(S_H(R(c_{[0]})_{(-1)3})) S_H(R(c_{[0]})_{(-1)4})
R(c_{[0]})_{(-1)1}B(R(c_{[0]})_{(-1)2})B(c_{[1]})\nonumber\\
&=R(c_{[0]})\o B(c_{[1]}).\label{3e}
\end{align}

Then
\begin{align*}
&R(c)_{1[0]}R(R(c)_{2(0)1[0]})R(S_C(R(c)_{2(0)3(0)}))\\
&\o B(R(c)_{1[1]1}) B(S_H(R(c)_{1[1]3})) R(c)_{2(-1)1} B(R(c)_{2(-1)2}) B(R(c)_{2(0)1[1]1})\\
&B(S_H(R(c)_{2(0)1[1]3}))B(S_H(R(c)_{2(0)2[1]}))\o R(c)_{2(0)2[0]}\o R(c)_{1[1]2} R(c)_{2(0)1[1]2} R(c)_{2(0)3(-1)}\\
&=R(c)_{1[0]}R(R(c)_{2(0)1[0]})R(S_C(R(c)_{2(0)3[0](0)}))
\o R(c)_{2(-1)1}B(R(c)_{2(-1)2})\\
&B(S_H(R(c)_{1[1]3})R(c)_{2(0)1[1]3})R(c)_{2(0)2[1]})R(c)_{2(0)3[1]2}))B(R(c)_{1[1]1})B(R(c)_{2(0)1[1]1})\cdot\\
&\cdot B(R(c)_{2(0)3[1]1})
\o R(c)_{2(0)2[0]}\o R(c)_{1[1]2}R(c)_{2(0)1[1]2}R(c)_{2(0)3[0](-1)}\\
&=R(c)_{1[0]}R(R(c)_{2(0)1[0]})R(S_C(R(c)_{2(0)3(0)[0]}))
\o R(c)_{2(-1)1}B(R(c)_{2(-1)2})\\
&B(S_H(R(c)_{1[1]3})R(c)_{2(0)1[1]3})R(c)_{2(0)2[1]})R(c)_{2(0)3(0)[1]2}))
B(R(c)_{1[1]1})B(R(c)_{2(0)1[1]1})\cdot\\
&\cdot B(R(c)_{2(0)3(0)[1]1})\o R(c)_{2(0)2[0]}\o R(c)_{1[1]2}R(c)_{2(0)1[1]2}R(c)_{2(0)3(-1)}\\
&=R(c)_{1[0]}R(R(c)_{2(0)1[0]})R(S_C(R(c)_{3(0)[0]}))
\o R(c)_{2(-1)1}R(c)_{3(-1)1}B(R(c)_{2(-1)2}R(c)_{3(-1)2})\\
&B(S_H(R(c)_{1[1]3})R(c)_{2(0)1[1]3})R(c)_{2(0)2[1]})R(c)_{3(0)[1]2}))
B(R(c)_{1[1]1})B(R(c)_{2(0)1[1]1})B(R(c)_{3(0)[1]1})\\
&\o R(c)_{2(0)2[0]}\o R(c)_{1[1]2}R(c)_{2(0)1[1]2}R(c)_{3(-1)3}\\
&=R(c)_{1[0]}R(R(c)_{2(0)[0]})R(S_C(R(c)_{4(0)[0]}))\\
&\o R(c)_{2(-1)1}R(c)_{3(-1)1}R(c)_{4(-1)1}B(R(c)_{2(-1)2}R(c)_{3(-1)2}R(c)_{4(-1)2})\cdot\\
&\cdot B(S_H(R(c)_{1[1]3})R(c)_{2(0)[1]3})R(c)_{3(0)[1]})R(c)_{4(0)[1]2}))
 B(R(c)_{1[1]1})B(R(c)_{2(0)[1]1})B(R(c)_{4(0)[1]1})\\
&\o R(c)_{3(0)[0]}\o R(c)_{1[1]2}R(c)_{2(0)[1]2}R(c)_{4(-1)3}\\
&=R(c)_{1[0](0)}R(R(c)_{2(0)[0][0]})R(S_C(R(c)_{4(0)[0](0)[0]}))\\
&\o \underline{R(c)_{1[0](-1)1}}R(c)_{2(-1)1}R(c)_{3(-1)1}R(c)_{4(-1)1}B(\underline{R(c)_{1[0](-1)2}}R(c)_{2(-1)2}R(c)_{3(-1)2}R(c)_{4(-1)2})\cdot\\
&\cdot B(S_H(R(c)_{1[1]3})R(c)_{2(0)[1]})R(c)_{3(0)[1]})R(c)_{4(0)[1]}))
\underline{S_H(R(c)_{1[0](-1)4})}B(\underline{S_H(R(c)_{1[0](-1)3}}))\cdot\\
&\cdot B(R(c)_{1[1]1})B(R(c)_{2(0)[0][1]1})B(R(c)_{4(0)[0](0)[1]})\\
&\o R(c)_{3(0)[0]}\o R(c)_{1[1]2}R(c)_{2(0)[0][1]2}R(c)_{4(0)[0](-1)}\\
&=R(c)_{1(0)[0](0)}R(R(c)_{2(0)[0][0]})R(S_C(R(c)_{4(0)[0](0)[0]}))\\
&\o R(c)_{1(-1)1}R(c)_{2(-1)1}R(c)_{3(-1)1}R(c)_{4(-1)1}B(R(c)_{1(-1)2}R(c)_{2(-1)2}R(c)_{3(-1)2}R(c)_{4(-1)2})\\
&B(S_H(R(c)_{1(0)[1]3})R(c)_{2(0)[1]})R(c)_{3(0)[1]})R(c)_{4(0)[1]}))
S_H(R(c)_{1(0)[0](-1)2})B(S_H(R(c)_{(0)1[0](-1)1}))\\
&B(R(c)_{1(0)[1]1})B(R(c)_{2(0)[0][1]1})B(R(c)_{4(0)[0](0)[1]})\\
&\o R(c)_{3(0)[0]}\o R(c)_{1(0)[1]2}R(c)_{2(0)[0][1]2}R(c)_{4(0)[0](-1)}\\
&=R(c)_{(0)1[0](0)}R(R(c)_{(0)2[0][0]})R(S_C(R(c)_{(0)4[0](0)[0]}))\\
&\o R(c)_{(-1)1}B(R(c)_{(-1)2})B(S_H(R(c)_{(0)1[1]3})R(c)_{(0)2[1]})R(c)_{(0)3[1]})R(c)_{(0)4[1]}))\\
&S_H(R(c)_{(0)1[0](-1)2})B(S_H(R(c)_{(0)[0]1[0](-1)1}))\\
&B(R(c)_{(0)1[1]1})B(R(c)_{(0)2[0][1]1})B(R(c)_{(0)4[0](0)[1]})\\
&\o R(c)_{(0)3[0]}\o R(c)_{(0)1[1]2}R(c)_{(0)2[0][1]2}R(c)_{(0)4[0](-1)}\\
&=R(c)_{(0)[0]1(0)[0]}R(R(c)_{(0)[0]2[0]})R(S_C(R(c)_{(0)[0]4(0)[0]}))\\
&\o R(c)_{(-1)1}B(R(c)_{(-1)2})B(S_H(R(c)_{(0)[1]}))S_H(R(c)_{(0)[0]1(-1)2})B(S_H(R(c)_{(0)[0]1(-1)1}))\\
&B(R(c)_{(0)[0]1(0)[1]1})B(R(c)_{(0)[0]2[1]1})B(R(c)_{(0)[0]4(0)[1]})\\
&\o R(c)_{(0)[0]3}\o R(c)_{(0)[0]1(0)[1]2}R(c)_{(0)[0]2[1]2}R(c)_{(0)[0]4(-1)}\\
&\stackrel{(\ref{3d})}{=}R(c_{[0]})_{1(0)[0]}R(\underline{R(c_{[0]})_{2[0][0]}})R(S_C(R(c_{[0]})_{4(0)[0]}))\\
&\o B(S_H(c_{[1]}))S_H(R(c_{[0]})_{1(-1)2})B(S_H(R(c_{[0]})_{1(-1)1}))\\
&B(R(c_{[0]})_{1(0)[1]1})\underline{B(R(c_{[0]})_{2[0][1]})}B(R(c_{[0]})_{4(0)[1]})\o R(c_{[0]})_{3}\\
&\o R(c_{[0]})_{1(0)[1]2}R(c_{[0]})_{2[1]}R(c_{[0]})_{4(-1)}\\
&\stackrel{(\ref{3e})}{=}R(c_{[0]})_{1(0)[0]}R(R(c_{[0]})_{2[0]})_{(0)[0]}\underline{R(S_C(R(c_{[0]})_{4(0)})_{[0]})}\\
&\o B(S_H(c_{[1]}))S_H(R(c_{[0]})_{1(-1)2})B(S_H(R(c_{[0]})_{1(-1)1}))B(R(c_{[0]})_{1(0)[1]1})\\
&B(S_H(R(R(c_{[0]})_{2[0]})_{(-1)1}))S_H(R(R(c_{[0]})_{2[0]})_{(-1)2})B(R(R(c_{[0]})_{2[0]})_{(0)[1]})\underline{B(S_C(R(c_{[0]})_{4(0)})_{[1]})}\\
&\o R(c_{[0]})_{3}\o R(c_{[0]})_{1(0)[1]2}R(c_{[0]})_{2[1]}R(c_{[0]})_{4(-1)}\\
&\stackrel{(\ref{3e})}{=}R(c_{[0]})_{1(0)[0]}R(R(c_{[0]})_{2[0]})_{(0)[0]}R(S_C(R(c_{[0]})_{4(0)}))_{(0)[0]}\\
&\o B(S_H(c_{[1]}))S_H(R(c_{[0]})_{1(-1)2})B(S_H(R(c_{[0]})_{1(-1)1}))B(R(c_{[0]})_{1(0)[1]1})\\
&B(S_H(R(R(c_{[0]})_{2[0]})_{(-1)1}))S_H(R(R(c_{[0]})_{2[0]})_{(-1)2})B(R(R(c_{[0]})_{2[0]})_{(0)[1]})\\
&B(S_H(R(S_C(R(c_{[0]})_{4(0)}))_{(-1)1}))S_H(R(S_C(R(c_{[0]})_{4(0)}))_{(-1)2})B(R(S_C(R(c_{[0]})_{4(0)}))_{(0)[1]})\\
&\o R(c_{[0]})_{3}\o R(c_{[0]})_{1(0)[1]2}R(c_{[0]})_{2[1]}R(c_{[0]})_{4(-1)}\\
&=R(c_{[0]})_{1[0](0)[0]}R(R(c_{[0]})_{2[0]})_{(0)[0]}R(S_C(R(c_{[0]})_{4(0)}))_{(0)[0]}\\
&\o B(S_H(c_{[1]}))S_H(R(c_{[0]})_{1[0](-1)2})B(S_H(R(c_{[0]})_{1[0](-1)1}))B(R(c_{[0]})_{1[0](0)[1]})\\
&B(S_H(R(R(c_{[0]})_{2[0]})_{(-1)1}))S_H(R(R(c_{[0]})_{2[0]})_{(-1)2})B(R(R(c_{[0]})_{2[0]})_{(0)[1]})\\
&B(S_H(R(S_C(R(c_{[0]})_{4(0)}))_{(-1)1}))S_H(R(S_C(R(c_{[0]})_{4(0)}))_{(-1)2})B(R(S_C(R(c_{[0]})_{4(0)}))_{(0)[1]})\\
&\o R(c_{[0]})_{3}\o R(c_{[0]})_{1[1]}R(c_{[0]})_{2[1]}R(c_{[0]})_{4(-1)}\\
&=R(c_{[0]})_{1[0]1(0)[0]}R(R(c_{[0]})_{1[0]2})_{(0)[0]}R(S_C(R(c_{[0]})_{3(0)}))_{(0)[0]}\\
&\o B(S_H(c_{[1]}))S_H(R(c_{[0]})_{1[0]1(-1)2})B(S_H(R(c_{[0]})_{1[0]1(-1)1}))B(R(c_{[0]})_{1[0]1(0)[1]})\\
&B(S_H(R(R(c_{[0]})_{1[0]2})_{(-1)1}))S_H(R(R(c_{[0]})_{1[0]2})_{(-1)2})B(R(R(c_{[0]})_{1[0]2})_{(0)[1]})\\
&B(S_H(R(S_C(R(c_{[0]})_{3(0)}))_{(-1)1}))S_H(R(S_C(R(c_{[0]})_{3(0)}))_{(-1)2})B(R(S_C(R(c_{[0]})_{3(0)}))_{(0)[1]})\\
&\o R(c_{[0]})_{2}\o R(c_{[0]})_{1[1]}R(c_{[0]})_{3(-1)}\\
&=[R(c_{[0]})_{1[0]1}R(R(c_{[0]})_{1[0]2})R(S_C(R(c_{[0]})_{3(0)}))]_{(0)[0]}\\
&\o B(S_H(c_{[1]}))B(S_H[R(c_{[0]})_{1[0]1}R(R(c_{[0]})_{1[0]2}R(S_C(R(c_{[0]})_{3(0)})))]_{(-1)1})\\
&S_H[R(c_{[0]})_{1[0]1}R(R(c_{[0]})_{1[0]2})R(S_C(R(c_{[0]})_{3(0)}))]_{(-1)2})\\
&B[R(c_{[0]})_{1[0]1}R(R(c_{[0]})_{1[0]2}R(S_C(R(c_{[0]})_{3(0)})))]_{(0)[1]})\\
&\o R(c_{[0]})_{2}\o R(c_{[0]})_{1[1]}R(c_{[0]})_{3(-1)}\\
&=R(c_{[0]1(0)})_{(0)[0]}\o B(S_H(c_{[1]}))B(S_H(R(c_{[0]1(0)})_{(-1)1})S_H(R(c_{[0]1(0)})_{(-1)2})B(R(c_{[0]1(0)})_{(0)[1]})\\
&\o R(c_{[0]2})\o B(S_H(c_{[0]1(-1)}))\\
&=R(c_{[0]1(0)[0]})\o B(S_H(c_{[1]}))B(c_{[0]1(0)[1]})\o R(c_{[0]2})\o B(S_H(c_{[0]1(-1)}))\\
&=R(c_{1[0](0)})\o B(S_H(c_{1[1]2}c_{2[1]}))B(c_{1[1]1})\o R(c_{2[0]})\o B(S_H(c_{1[0](-1)}))\\
&=R(c_{1(0)})\o B(S_H(c_{2[1]}))\o R(c_{2[0]})\o B(S_H(c_{1(-1)})).
\end{align*}
The proof is completed.
\end{proof}

When the right $H$-coaction on $C$ is trivial, that is, $c_{[0]}\o c_{[1]}=c\o 1_H$, we could obtain the following results.
\begin{corollary}
Let $C$ and $H$ be two commutative Hopf algebra such that $(C,\r^l)$ is a left $H$-comodule Hopf algebra. Assume that $R:C\rightarrow C$ is a linear map and $B$ is a Rota-Baxter co-operator on $H$. Define a linear map $\widetilde{B}:C\ltimes H\rightarrow C\ltimes H$ by
$$\widetilde{B}(c\ltimes h)=R(c_{(0)})\ltimes B(c_{(-1)}h),$$
for all $c\in C,h\in H$. Then $\widetilde{B}$ is a Rota-Baxter co-operator on $C\ltimes H$ if and only if $R$ is a Rota-Baxter co-operator on $C$ and satisfies the following conditions
\begin{align}
&R(c)_{1}R(R(c)_{2}S_C(R(c)_{3(0)}))\o R(c)_2\o R(c)_{3(-1)}\nonumber\\
&\ =R(c_{1(0)})\o R(c_2)\o B(S_H(c_{1(-1)})),\label{3f}\\
&R(c)_{(-1)1}B(R(c)_{(-1)2}\o R(c)_{(0)}=1_H\o R(c).\label{3g}
\end{align}
\end{corollary}

In Theorem  \ref{3-1}, when the Rota-Baxter co-operator on $C$ and $H$ are replaced by $S_C$ and $S_H$ respectively, we could get the following results directly.
\begin{corollary}
Let $C$ and $H$ be two commutative Hopf algebra such that $(C,\r^l,\r^r)$ is an $H$-bicomodule Hopf algebra. Assume that $B$ is a Rota-Baxter co-operator on $H$. Define a linear map $\widetilde{B}:C\ltimes H\rightarrow C\ltimes H$ by
$$\widetilde{B}(c\ltimes h)=S_C(c_{(0)[0]})\ltimes B(c_{(-1)}c_{(0)[1]}h),$$
for all $c\in C,h\in H$. Then $\widetilde{B}$ is a Rota-Baxter co-operator on $C\ltimes H$ if and only if $R$ is a Rota-Baxter co-operator on $C$ and satisfies the following conditions
\begin{align}
&S_C(c_{1(0)})\o c_{1(-1)}
=S_C(c_{1(0)})\o B(S_H(c_{1(-1)})),\label{3h}\\
&c_{(-1)1}B(c_{(-1)2})\o c_{(0)}=1_H\o c.\label{3i}
\end{align}
\end{corollary}

\begin{corollary}
Let $C$ and $H$ be two commutative Hopf algebra such that $(C,\r^l,\r^r)$ is an $H$-bicomodule Hopf algebra. Assume that $R:C\rightarrow C$ is a linear map. Define a linear map $\widetilde{B}:C\ltimes H\rightarrow C\ltimes H$ by
$$\widetilde{B}(c\ltimes h)=R(c_{(0)[0]})\ltimes S_H(c_{(-1)}c_{(0)[1]}h),$$
for all $c\in C,h\in H$. Then $\widetilde{B}$ is a Rota-Baxter co-operator on $C\ltimes H$ if and only if $R$ is a Rota-Baxter co-operator on $C$ and satisfies the following conditions
\begin{align}
&R(c)_{1[0]1}R(R(c)_{1[0]2}S_C(R(c)_{3(0)}))\o R(c)_2\o R(c)_{1[1]}R(c)_{3(-1)}\nonumber\\
&\ =R(c_{1(0)})\o R(c_2)\o c_{1(-1)},\label{3j}\\
&R(c)_{[1]}\o R(c)_{[0]}=c_{[1]}\o R(c_{[0]}).\label{3k}
\end{align}
\end{corollary}

\begin{example}
Let $C_2$ and $C_3$ be the cyclic groups given in Example \ref{2-6}. Denote $k^{C_2}$ and $k^{C_3}$ the linear dual of $k[C_2]$ and $k[C_3]$, respectively. Then $k^{C_2}$ has a basis $p_1,p_g$ dual to $1,g$, and $k^{C_3}$ has a basis $q_1,q_h,q_{h^2}$ dual to $1,h,h^2$. Define the $k^{C_2}$-bicomodule action on $k^{C_3}$ by
\begin{align*}
&q_{1(-1)}\o q_{1(0)}=(p_1+p_g)\o q_1,\ q_{1[0]}\o q_{1[1]}=q_1\o (p_1+p_g),\\
&q_{h(-1)}\o q_{h(0)}=(p_1+p_g)\o q_h,\ q_{h[0]}\o q_{h[1]}=q_h\o p_1+q_{h^2}\o p_g,\\
&q_{h^2(-1)}\o q_{h^2(0)}=(p_1+p_g)\o q_{h^2},\ q_{h^2[0]}\o q_{h^2[1]}=q_{h^2}\o p_1+q_{h}\o p_g,
\end{align*}
then $k^{C_3}$ is a $k^{C_2}$-bicomodule coalgebra.
\begin{itemize}
  \item [(1)] Define $R:k^{C_3}\rightarrow k^{C_3}$ by $R(x)=x$ for all $x\in k^{C_3}$. Then
$$\widetilde{B}(q_{h^i}\ltimes p_{g^j})=1\ltimes p_{g^j},\ i=0,1,2;j=0,1.$$
is a Rota-Baxter co-operator on $k^{C_3}\ltimes k^{C_2}$.
  \item [(2)] Define $R:k^{C_3}\rightarrow k^{C_3}$ by 
$$R(q_1)=q_1,R(q_h)=q_{h^2},R(q_{h^2})=q_h.$$
Then $\widetilde{B}$ is the identity map on $k^{C_3}\ltimes k^{C_2}$, which is a Rota-Baxter co-operator since $k^{C_3}\ltimes k^{C_2}$ is cocommutative.

\end{itemize}
\end{example}

\section*{Acknowledgement}

This work was supported by the NNSF of China (Nos. 12271292, 11901240).

\end{document}